\documentclass{amsart}
\usepackage[foot]{amsaddr}
\usepackage{amssymb,amsmath,amsfonts,mathptmx,cite,mathrsfs,url}
\usepackage[mathscr]{eucal}
\DeclareSymbolFont{rsfscript}{OMS}{rsfs}{m}{n}
\DeclareSymbolFontAlphabet{\mathrsfs}{rsfscript}

\DeclareMathOperator{\Styl}{Styl}
\DeclareMathOperator{\Kis}{Kis}
\DeclareMathOperator{\Cat}{Cat}

\makeatletter

\renewcommand*\subjclass[2][2010]{\def\@subjclass{#2}\@ifundefined{subjclassname@#1}{\ClassWarning{\@classname}{Unknown edition (#1) of Mathematics Subject Classification; using '2010'.}}{\@xp\let\@xp\subjclassname\csname subjclassname@#1\endcsname}}

\renewcommand{\subjclassname}{\textup{2010} Mathematics Subject Classification}

\makeatother

\newtheorem{theorem}{Theorem}

\newtheorem{lemma}[theorem]{Lemma}
\newtheorem{corollary}[theorem]{Corollary}

\begin{document}

\title{Identities of the stylic monoid}
\thanks{Supported by the Ministry of Science and Higher Education of the Russian Federation (Ural Mathematical Center project No. 075-02-2022-877)}

\author[M. V. Volkov]{Mikhail V. Volkov}
\address{Institute of Natural Sciences and Mathematics\\
Ural Federal University\\ 620000 Ekaterinburg, Russia}
\email{m.v.volkov@urfu.ru}

\date{}

\begin{abstract}
We observe that for each $n\ge 2$, the identities of the stylic monoid with $n$ generators coincide with the identities of $n$-generated monoids from other distinguished series of $\mathrsfs{J}$-trivial monoids studied in the literature, e.g., Catalan monoids and Kiselman monoids. This solves the Finite Basis Problem for stylic monoids.
\end{abstract}



\maketitle

A \emph{monoid identity} is a pair of \emph{words}, i.e., elements of the free monoid $X^*$ over an alphabet $X$, written as a formal equality. An identity $w=w'$ with $w,w'\in X^*$ is said to \emph{hold in a monoid $M$} if $w\varphi=w'\varphi$ for each homomorphism $\varphi\colon X^*\to M$; alternatively, we say that the monoid \emph{satisfies} $w=w'$. Clearly, if $M$ satisfies $w=w'$, then so does every homomorphic image of $M$.

Given any set $\Sigma$ of monoid identities, we say that an identity $w=w'$ \emph{follows} from $\Sigma$ if every monoid satisfying all identities of $\Sigma$ satisfies the identity $w=w'$ as well. A subset $\Delta\subseteq\Sigma$ is called a \emph{basis} for $\Sigma$ if each identity in $\Sigma$ follows from $\Delta$. The \emph{Finite Basis Problem} for a monoid $M$ is the question of whether or not the set of all identities that hold in $M$ admits a finite basis.

A monoid $M$ is said to be $\mathrsfs{J}$-\emph{triv\-i\-al} if every principal ideal of $M$ has a unique generator, that is, $MaM=MbM$ implies $a=b$ for all $a,b\in M$. Finite $\mathrsfs{J}$-trivial monoids attract much attention because of their distinguished role in algebraic theory of regular languages~\cite{Si72,Si75} and representation theory~\cite{DHST11}. Several series of finite $\mathrsfs{J}$-trivial monoids parameterized by positive integers appear in the literature, including Straubing monoids~\cite{St80,Vo04}, Catalan monoids~\cite{So96}, double Catalan monoids~\cite{MS12}, Kiselman monoids~\cite{Ki02,KM09,GM11}, and gossip monoids~\cite{BDF15,FJK18,JF19}. These monoids arise due to completely unrelated reasons and consist of elements of a very different nature. Surprisingly, studying identities of the listed monoids has revealed that the $n$-th monoids in each series satisfy exactly the same identities!

Recently, a new family of finite $\mathrsfs{J}$-trivial monoids, coined \emph{stylic} monoids, have been introduced by Abram and Reutenauer~\cite{AR22}, with motivation coming from combinatorics of Young tableaux. It is quite natural to ask whether the above phenomenon extends to this new family, i.e., whether the $n$-th stylic monoid again satisfies the same identities as do the $n$-th monoids in each aforementioned series. The present note aims to answers this question in the affirmative.

The combinatorial definition of stylic monoids can be found in \cite{AR22}, but here we only need their presentation via generators and relations established in \cite[Theorem 8.1(ii)]{AR22}. Thus, let the stylic monoid $\Styl_n$ be the monoid generated by $a_1,a_2,\dots,a_n$ subject to the relations
\begin{align}
&a_i^2=a_i                      &&\text{for each } i=1,\dots,n;\label{eq:idempotent}\\
&a_ja_ia_k=a_ja_ka_i            &&\text{if } 1\le i<j<k\le n;\label{eq:plactic1}\\
&a_ia_ka_j=a_ka_ia_j            &&\text{if } 1\le i<j<k\le n;\label{eq:plactic2}\\
&a_ja_ia_i=a_ia_ja_i            &&\text{if } 1\le i<j\le n;\label{eq:plactic3}\\
&a_ja_ja_i=a_ja_ia_j            &&\text{if } 1\le i<j\le n.\label{eq:plactic4}
\end{align}
We want to relate $\Styl_n$ to two other monoids with same generating set. The Kiselman monoid $\Kis_n$, defined by Kiselman~\cite{Ki02} for $n=3$ and by Ganyushkin and Mazorchuk (unpublished) for an arbitrary $n\ge 2$, is generated by  $a_1,a_2,\dots,a_n$ subject to the relations
\begin{align}
&a_i^2=a_i                       &&\text{for each } i=1,\dots,n;\label{eq:idempotent1}\\
& a_ia_ja_i=a_ja_ia_j=a_ja_i     &&\text{if } 1\le i<j\le n.\label{eq:kiselman}
\end{align}
Catalan monoids were defined by Solomon~\cite{So96} as monoids of certain transformations on directed paths, but again, we only need their presentation via generators and relations from \cite[Section 9]{So96}. So, for the purpose of this note, let $\Cat_n$ stand for the monoid generated by  $a_1,a_2,\dots,a_n$ subject to the relations
\begin{align}
&a_i^2=a_i                  &&\text{for each } i=1,\dots,n;\label{eq:idempotent2}\\
&a_{i}a_{k}=a_{k}a_{i}      &&\text{if }|i-k|\ge 2,\ i,k=1,\dots,n;\label{eq:catalan1}\\
&a_ia_{i+1}a_i=a_{i+1}a_ia_{i+1}=a_{i+1}a_i &&\text{for each } i=1,\dots,n-1.\label{eq:catalan2}
\end{align}

\begin{lemma}
\label{lem:Kis-mapsonto-Styl}
The relations \eqref{eq:kiselman} hold in the monoid $\Styl_n$.
\end{lemma}

\begin{proof}
For any $i,j\in\{1,2,\dots,n\}$ with $i<j$, we can deduce in $\Styl_n$ the following:
\[
a_ia_ja_i\stackrel{\eqref{eq:plactic3}}{=}a_ja_ia_i\stackrel{\eqref{eq:idempotent}}{=}a_ja_i\quad\text{and}\quad a_ja_ia_j\stackrel{\eqref{eq:plactic4}}{=}a_ja_ja_i\stackrel{\eqref{eq:idempotent}}{=}a_ja_i.
\]
Hence the relation $a_ia_ja_i=a_ja_ia_j=a_ja_i$ holds in $\Styl_n$.
\end{proof}

\begin{lemma}
\label{lem:Styl-mapsonto-Cat}
The relations \eqref{eq:plactic1}--\eqref{eq:plactic4} hold in the monoid $\Cat_n$.
\end{lemma}

\begin{proof}
If $i,j,k\in\{1,2,\dots,n\}$ are such that $i<j<k$, then $k-i\ge2$ whence $a_ja_ia_k\stackrel{\eqref{eq:catalan1}}{=}a_ja_ka_i$ and $a_ia_ka_j\stackrel{\eqref{eq:catalan1}}{=}a_ka_ia_j$ in $\Cat_n$. Thus, \eqref{eq:plactic1} and \eqref{eq:plactic2} hold in $\Cat_n$.

If $1\le i<j\le n$ and $j-i\ge2$, we can apply \eqref{eq:catalan1} in a similar way: $a_ja_ia_i\stackrel{\eqref{eq:catalan1}}{=}a_ia_ja_i$ and $a_ja_ja_i\stackrel{\eqref{eq:catalan1}}{=}a_ja_ia_j$. Therefore, to prove that \eqref{eq:plactic3} and \eqref{eq:plactic4}  hold in $\Cat_n$, it remains to consider the case $j=i+1$. In this case, we can deduce in $\Cat_n$ the following:
\[
a_{i+1}a_ia_i\stackrel{\eqref{eq:idempotent2}}{=}a_{i+1}a_i\stackrel{\eqref{eq:catalan2}}{=}a_{i+1}a_ia_{i+1}\quad\text{and}\quad
a_{i+1}a_{i+i}a_i\stackrel{\eqref{eq:idempotent2}}{=}a_{i+1}a_i\stackrel{\eqref{eq:catalan2}}{=}a_ia_{i+1}a_i.
\]
Hence  \eqref{eq:plactic3} and \eqref{eq:plactic4} also hold in $\Cat_n$ for all $i,j$ with $1\le i<j\le n$.
\end{proof}

\begin{theorem}
\label{thm:coincide}
For each $n\ge 2$, the monoids $\Cat_n$, $\Styl_n$, and  $\Kis_n$ satisfy the same identities.
\end{theorem}

\begin{proof}
We invoke Dyck's Theorem (see, e.g., \cite[Theorem III.8.3]{Cohn}). Specialized in the case of monoids, it says that if a monoid $M$ is generated by a set $A$ subject to relations $R$ and $N$ is a monoid generated by $A$ and such that all the relations $R$ hold in $N$, then $N$ is a homomorphic image of $M$. In view of this fact, Lemmas~\ref{lem:Styl-mapsonto-Cat} and~\ref{lem:Kis-mapsonto-Styl} imply that $\Cat_n$ is a homomorphic image of $\Styl_n$, which in turn is a homomorphic image of $\Kis_n$. Since identities are inherited by homomorphic images, every identity holding in $\Kis_n$ holds in $\Styl_n$, and every identity holding in $\Styl_n$ holds in $\Cat_n$. However, it follows from \cite[Theorem~8]{AVZ15} that for each $n\ge 2$, the monoids $\Cat_n$ and  $\Kis_n$ satisfy the same identities. Hence the same identities hold in the `intermediate' monoid $\Styl_n$ as well.
\end{proof}

Since the identities of the monoids $\Cat_n$ and  $\Kis_n$ have been characterized in \cite{Vo04,AVZ15}, Theorem~\ref{thm:coincide} leads to an efficient combinatorial description of the identities of stylic monoids. The description involves the notion of a scattered subword. Recall that a product $x_1\cdots x_k$ of elements from an alphabet $X$ is said to be a \emph{scattered subword of length $k$} in a word $v\in X^*$ if there exist words $v_0,v_1,\dots,v_{k-1},v_k\in X^*$ such that $v=v_0x_1v_1\cdots v_{k-1}x_kv_k$; in other terms, $x_1\cdots x_k$ is as a subsequence of $v$. The following is a combination of Theorem~\ref{thm:coincide} with either~\cite[Proposition~4 and Corollary~2]{Vo04} or \cite[Theorem~8]{AVZ15}.

\begin{corollary}
\label{cor:scattered}
An identity $w=w'$ holds in the monoid $\Styl_n$ if and only if the words $w$ and $w'$ have the same set of scattered subwords of length at most $n$.
\end{corollary}

As yet another immediate application, we get a solution to the Finite Basis Problem for stylic monoids. It comes from Corollary~\ref{cor:scattered} combined with results by Blanchet-Sadri~\cite{Bl93,Bl94}.
\begin{corollary}
\label{cor:fbp for Styln}
\emph{a)} The identities $xyxzx=xyzx,\ (xy)^2=(yx)^2$ form a basis for the identities of the monoid $\Styl_2$.

\emph{b)} The identities $xyx^2zx= xyxzx,\ xyzx^2tz= xyxzx^2tz,\ zyx^2ztx= zyx^2zxtx,\ (xy)^3= (yx)^3$ form a basis for the identities of the monoid $\Styl_3$.

\emph{c)} The identities of the monoid $\Styl_n$ with $n\ge4$ admit no finite basis.
\end{corollary}


\begin{thebibliography}{99}
\bibitem{AR22}
A. Abram, C. Reutenauer, The stylic monoid, Semigroup Forum 105 (2022), \url{https://doi.org/10.1007/s00233-022-10285-3} 

\bibitem{AVZ15}
D. N. Ashikhmin, M. V. Volkov, Wen Ting Zhang, The finite basis problem for Kiselman monoids, Demonstr. Math. 48 (2015), no. 4, 475--492.

\bibitem{Bl93}
F. Blanchet-Sadri, Equations and dot-depth one, Semigroup Forum 47 (1993), no.3, 305--317.

\bibitem{Bl94}
F. Blanchet-Sadri, Equations and monoid varieties of dot-depth one and two, Theor. Comp. Sci. 123 (1994), no.2, 239--258.

\bibitem{BDF15}
A. E. Brouwer, J. Draisma, B. J. Frenk, Lossy gossip and composition of metrics. Discrete Comput. Geom. 53 (2015) no. 4, 890--913.

\bibitem{Cohn}
P. M. Cohn, Universal Algebra,  Mathematics and Its Applications, vol. 6. D. Reidel Publishing Company, Dordrecht (1981)

\bibitem{DHST11}
T. Denton, F. Hivert, A. Schilling, N. M. Thi\'ery, On the representation theory of finite $\mathcal{J}$-trivial monoids, S\'eminaire Lotharingien de Combinatoire 64 (2011), Article no. B64d.

\bibitem{FJK18}
P. Fenner, M. Johnson, M. Kambites, NP-completeness in the gossip monoid, Int. J. Algebra Comput. 28 (2018), no.4, 653-–672.

\bibitem{GM11}
O. Ganyushkin, V. Mazorchuk, On Kiselman quotients of 0-Hecke monoids, Int. Electron. J. Algebra 10 (2011), no. 2, 174--191.

\bibitem{JF19}
M. Johnson, P. Fenner, Identities in unitriangular and gossip monoids, Semigroup Forum 98 (2019), no. 2, 338--354.

\bibitem{Ki02}
Ch. O. Kiselman, A semigroup of operators in convexity theory, Trans. Amer. Math. Soc. 354 (2002), no. 5, 2035--2053.

\bibitem{KM09}
G. Kudryavtseva, V. Mazorchuk, On Kiselman's semigroup, Yokohama Math. J. 55 (2009), no. 1, 21--46.

\bibitem{MS12}
V. Mazorchuk, B. Steinberg, Double Catalan monoids, J. Algebraic Combin. 36 (2012) no. 3, 333-–354.

\bibitem{Si72}
I. Simon, Hierarchies of Events of Dot-Depth One, Ph.D. thesis, University of Waterloo, 1972.

\bibitem{Si75}
I. Simon, Piecewise testable events, Proc. 2nd GI Conf. [Lect. Notes Comput. Sci., vol.33], Springer, Berlin, 1975, 214--222.

\bibitem{So96}
A. Solomon, Catalan monoids, monoids of local endomorphisms, and their presentations, Semigroup Forum, 53 (1996), no. 3, 351--368.

\bibitem{St80}
H. Straubing, On finite $\mathrsfs{J}$-trivial monoids, Semigroup Forum 19 (1980), no. 2, 107--110.

\bibitem{Vo04}
M. V. Volkov, Reflexive relations, extensive transformations and piecewise testable languages of a given height, Int. J. Algebra Comput. 14 (2004), no. 5--6, 817--827.
\end{thebibliography}
\end{document}